\documentclass{amsart}
\usepackage{amssymb}

\newtheorem{theorem}{Theorem}[section]
\newtheorem{lemma}[theorem]{Lemma}
\newtheorem{cor}[theorem]{Corollary}

\theoremstyle{definition}
\newtheorem{definition}[theorem]{Definition}

\theoremstyle{remark}
\newtheorem{remark}[theorem]{Remark}

\begin{document}

\title{The growth of torus link groups}

\author{Yoshiyuki Nakagawa}
\address{Department of Economics, Ryukoku University, Kyoto, 612-8577, Japan}
\email{nakagawa@mail.ryukoku.ac.jp}

\author{Makoto Tamura}
\address{College of General Education,
         Osaka Sangyo University, Osaka, 574-8530, Japan}
\email{mtamura@las.osaka-sandai.ac.jp}

\author{Yasushi Yamashita}
\address{Department of Information and Computer Sciences,
         Nara Women's University, Nara, 630-8506, Japan}
\email{yamasita@ics.nara-wu.ac.jp}
\thanks{This work was supported by JSPS KAKENHI Grant Number 23540088}



\date{\today}

\begin{abstract}
Let $G$ be a finitely generated group with a finite generating set $S$.
For $g\in G$,
let $l_S(g)$ be the length of the shortest word over $S$ representing $g$.
The growth series of $G$ with respect to $S$ is
the series $A(t) = \sum_{n=0}^\infty a_n t^n$,
where $a_n$ is the number of elements of $G$ with $l_S(g)=n$.
If $A(t)$ can be expressed as a rational function of $t$,
then $G$ is said to have a rational growth function.

We calculate explicitly the rational growth functions of 
$(p,q)$-torus link groups for any $p, q > 1.$
As an application, we show that their growth rates are Perron numbers.
\end{abstract}

\maketitle

\section{Introduction and the main result}
\label{sec:intro}

Let $G$ be a finitely generated group
with a finite generating set $S$.
For $g\in G$, we denote by $l_S(g)$ 
the length of the shortest word over $S$ representing $g$.
The growth series of $G$ with respect to $S$ is 
the series $A(t) = \sum_{n=0}^\infty a_n t^n$,
where $a_n$ is the number of elements of $G$ with $l_S(g)=n$.
In many important and interesting cases,
$A(t)$ can be expressed as a rational function of $t$.
However, it is often very difficult to compute $A(t)$ explicitly.
Note that the radius of convergence $R$ of $A(t)=P(t)/Q(t)$
is the modulus of a zero of the denominator $Q(t)$
closest to the origin $0\in\mathbb C$ of all zeros of $Q(t)$.
Recall that the growth rate 
$\omega = \limsup_{n\rightarrow\infty} \sqrt[n]{a_n}$ is
equal to $1/R$.
After Cannon-Wagreich \cite{MR1166120},
nice algebraic properties of $R$ or $\omega$
(e.g., Salem numbers, Pisot numbers, Perron numbers) 
were investigated. 

In this paper, we consider torus link group:
the fundamental group of the complement of a torus link in 3-sphere.
They are important objects in knot theory.
The $(p, q)$-torus link group has a presentation 
$T_{p,q} = \langle x, y | x^p = y^q \rangle.$
$T_{p,q}$ has a nontrivial center generated by $x^p = y^q$
and we add this element, say $z$, to the generating set.
We denote this presentation
by $T_{p,q}' = \langle x, y, z | x^p = y^q = z\rangle$.
The case for $p=2$ and $q=3$ (the trefoil group) has some history.
In \cite{MR1285316}, $T_{2,3}$ and $T_{2,3}'$
were studied as a three stranded braid group and 
the rational growth functions for both cases were given.
Also, the rational growth function for $T_{2,3}'$
was calculated in \cite{MR1476956} and in chapter 14 of \cite{MR2894945},
which is the starting point of this paper.
We note that $T_{2,3}$ and $T_{2,4}$
were studied in \cite{MR2746061}
as examples of groups for which the growth rates are realized.
(See \cite{MR2746061} for the definition of ``realized.'')

The aim of this paper is to give the explicit expression 
of the rational growth function of $T_{p,q}'$ for any $p, q>1$,
and study the algebraic property of the growth rate.
We denote by $C_n(t)$ the growth series (polynomial) of
the cyclic group of order $n$, i.e., 
$C_n(t) = 1 + 2t + \dots + 2t^{(n-1)/2}$ if $n$ is odd, and 
$C_n(t) = 1 + 2t + \dots + 2t^{n/2-1} + t^{n/2}$ if $n$ is even,
and $C_\infty(t) = (1+t)/(1-t)$ for infinite cyclic case.
Here is our main theorem:

\begin{theorem}
The rational growth function of 
$T_{p,q}' = \langle x, y, z | x^p = y^q = z\rangle$ for any $p,q>1$ is
\[
A(t)  = \frac{ C_\infty(t) C_p(t) C_q(t)}{C_p(t)+C_q(t)-C_p(t)C_q(t)}
 + \frac{ m_q(t) C_p(t)^2 + m_p(t) C_q(t)^2}{(C_p(t)+C_q(t)-C_p(t)C_q(t))^2},
\]
where $m_r(t) = t^{r/2}$ if $r$ is even, $0$ if $r$ is odd. 
\end{theorem}

The proof will be given in the next section.

\begin{remark}
The rational growth function of 
$\langle x, y | x^p = y^q \rangle$ was calculated by Gill.
See Theorem 2.3.2 in \cite{MR1695317}.
(It was pointed out that it has a misprint in the formula.
 See \cite{MR2746061} \cite{MR2894945}.)
\end{remark}

\begin{remark}
Our result should have a generalization to the groups
\[
\langle x_1, x_2, \dots, x_r, z \:|\:
        x_1^{p_1} = x_2^{p_2} = \dots = x_r^{p_r} = z \rangle,
\]
where $2 \leq p_1 \leq p_2 \leq \dots \leq p_r$.
See \cite{MR1695317}.
For example, if $p_i$ is odd for all $i$,
we can apply Theorem \ref{thm:fpa} in the next section one by one.
\end{remark}

Here is a direct consequence of the main result.
A real number $\omega$ is called 
a Perron number if it is greater than 1 and
an algebraic integer
whose conjugates have moduli less than the modulus of $\omega$.

\begin{cor}\label{cor:gr}
The growth rate of $T_{2,2}'$ is $1$.
If $(p, q)\neq(2, 2)$, 
the growth rate of $T_{p,q}'$ is a Perron number.
\end{cor}

To prove this corollary, let us recall a result 
by Komori-Umemoto \cite{MR2912844}.
To show that the growth rates 
of three-dimensional non-compact hyperbolic Coxeter groups with four generators
are Perron numbers,
they proved the next lemma:
\begin{lemma}[Lemma 1 in \cite{MR2912844}]
Consider the polynomial of degree $n\geq 2$
\[
  g(t) = \sum_{k=1}^n a_k t^k - 1,
\]
where $a_k$ is a non-negative integer.
We also assume that the greatest common divisor of 
$\{k\in{\mathbb N}\: | \: a_k\neq 0 \}$ is $1$.
Then there exists a real number $r_0$, $0<r_0<1$
which is the unique zero of $g(t)$ having the smallest absolute value
of all zeros of $g(t)$.
\end{lemma}

\begin{proof}[Proof of Corollary~\ref{cor:gr}]
The growth rate of $T_{2,2}'$ can be calculated directly.

Suppose that $(p, q)\neq (2,2)$.
By considering the natural map from $T_{p,q}'$ 
to $\langle x,y,z | x^2=y^q=z=id \rangle \simeq {\mathbb Z}_p * {\mathbb Z}_q$,
we see that the growth rate $\omega$ is greater than $1$.
(See Theorem 16.12 in \cite{MR2894945}.)
Hence, the radius of convergence $R$ of $A(t)$ is smaller than $1$.
Put $g(t) = (C_p(t)-1)(C_q(t)-1)-1$.
Clearly, $g(t)$ satisfies the condition of the lemma.
By the main theorem, the set of poles of $A(t)$ other than $1$ is contained 
in the set of zeros of $g(t)$.
Thus, $R$ is equal to the unique real zero of $g(t)=0$ between $0$ and $1$.
Then, $\omega$ is a zero of the reciprocal $g^*(t)$ of $-g(t)$
which is monic since the constant term of $-g(t)$ is $1$,
and zeros of $g^*(t)$ have moduli less than the modulus of $\omega$.
This completes the proof.
\end{proof}

\section{Proof}
\label{sec:growth}

In this section, we prove the main theorem.
What we will do is
to follow the arguments in chapter 14 of \cite{MR2894945} carefully,
apply the same idea to the general cases with some modification,
and carry out the calculations.

Let $p, q$ be integers greater than one.
Let $X_p = \langle x, z | x^p = z \rangle$ and
$Y_q = \langle y, z | y^q = z \rangle$.
They are infinite cyclic groups with one extra generator added.
We say that a shortest word in $X_p$ (resp. $Y_q$)
of the form $z^i x^j$ (resp. $z^i y^j$) a normal form.

We separate the proof into three cases:
(1) $p$ and $q$ are both odd, 
(2) $p$ is even and $q$ is odd,
(3) $p$ and $q$ are both even.

\subsection{$p$ and $q$ are both odd}

To our surprise, this case was very easy.
We can apply a general theorem for free product with amalgamation.
Let us recall notation and a theorem.
See \cite{MR2894945} for detail.

\begin{definition}[Section 14.3 in \cite{MR2894945}]
Consider pairs $(G, S)$, where $G$ is generated by $S$.
A pair $(H, T)$ is admissible in $(G, S)$, if $H$ is a subgroup of $G$,
$T$ is a subset of $S$, and there exists a transversal $U$ for $H$ in $G$,
termed an admissible transversal,
such that if $g=hu$ with $g\in G$, $h\in H$, $u\in U$,
then $l_S(g) = l_T(h) + l_S(u)$.
We always assume that the transversal contains the identity
as the representative for $H$.
\end{definition}

\begin{theorem}[Prop.\ 14.2 in \cite{MR2894945}]
\label{thm:fpa}
Let $(L,R)$ be admissible in both $(H, S)$ and $(K, T)$.
Let $G = H *_L K$.
Let $G$, $H$, $K$, $L$ have
growth functions $A(t)$, $B(t)$, $C(t)$ and $D(t)$, respectively,
relative to the generating set $S\cup T$, $S$, $T$, $R$.
Then we have
\[
\frac {1}{A(t)} = \frac {1}{B(t)} + \frac {1}{C(t)} - \frac {1}{D(t)}.
\]
\end{theorem}

Now, let us begin the proof for this case.
Suppose that $p=2n+1$ and $q=2m+1$.
Note that
$T_{p,q}' = 
\langle x, y, z | x^p = y^q = z \rangle = X_p *_{\langle z\rangle} Y_q$
is a free product with amalgamation.

Then, $(\langle z \rangle, \{z\})$ is admissible in both
$(X_p, \{x,z\})$ and $(Y_q, \{y,z\})$.
To see this, set 
$U_p = \{ x^{-n}, x^{-n+1}, \dots, x^{n-1}, x^n \}$
and 
$V_q = \{ y^{-m}, y^{-m+1}, \dots, y^{p-1}, y^m \}$.
Then, $U_p$ (resp. $V_q$) is an admissible transversal for 
$\langle z \rangle$ in $X_p$ (resp. $Y_q$).

Hence, we can apply Theorem~\ref{thm:fpa}.
Let $B(t), C(t)$ and $D(t)$ be the growth functions for $X_p, Y_q$ and 
$\langle z \rangle$, respectively.
Then, it is easy to see that 
$B(t) = C_\infty(t) C_p(t)$, $C(t) = C_\infty(t) C_q(t)$,
and $D(t)=C_\infty(t)$.
By Theorem~\ref{thm:fpa}, we have
\[
 A(t) = \frac{C_\infty(t) C_p(t) C_q(t)}{C_p(t)+C_q(t)-C_p(t)C_q(t)}
\]
which gives the claimed formula.

\subsection{$p$ is even and $q$ is odd}

Suppose that $p=2n$ and $q=2m+1$.
Unfortunately, 
$(\langle z \rangle, \{z\})$ is not admissible in 
$(X_p, \{x,z\})$.
Set 
$U_p = \{ x^{-n}, x^{-n+1}, \dots, x^{n-1}, x^n \}$
and 
$V_q = \{ y^{-m}, y^{-m+1}, \dots, y^{p-1}, y^m \}$ as in the previous case.
Note that $V_q$ is an admissible transversal for 
$\langle z \rangle$ in $Y_q$, but 
$U_p$ is not a transversal for
$\langle z \rangle$ in $X_p$
because $x^n = z x^{-n}$.

\subsubsection{The minimal normal form}

(See section 14.1 of \cite{MR2894945}.)

Let $w$ be a shortest word in $T_{p,q}'$.
Collecting together generators from the same group,
we obtain a product in which elements from $X_p$ and $Y_q$ alternate.
We replace each factor by its normal form in $X_p$ or $Y_q$.
Then, we get a product in which elements from $\langle z \rangle$ alternate
with elements from $U_p$ and $V_q$.
If we have in this product a segment of the form $x z^j$, say, 
then $x z^j \in X_p$, and we can replace it by its normal form 
$z^{j'} x'$, without increasing the length.
In this way, we represent each element in $G$ 
by a minimal word of the form $w=z^i u$
where the occurrences of elements in $U_p$ and $V_q$ alternate in $u$,

Suppose that $i>0$ and 
the word $x^{-n}$ occurs in $u$.
Then, we can replace $x^{-n}$ by $z^{-1} x^n$ and 
send this $z^{-1}$ to the left and replace $z^i$ by $z^{i-1}$.
The case for $i<0$ and $x^n$ is similar.
Therefore, we may assume that,
if $i\neq 0$, then in all occurrences of $x^{\pm n}$,
the exponent has the same sign, which is the sign of $i$.

Next, suppose that $i=0$.
If $u$ contains both $x^n$ and $x^{-n}$,
then we can write $x^{-n} = z^{-1} x^n$,
then move $z^{-1}$ to the position of $x^n$,
and replace this $z^{-1} x^n$ by $x^{-n}$.
This means that we can interchange the occurrences of $x^n$ and $x^{-n}$,
therefore we may assume that all terms $x^n$ proceed all terms $x^{-n}$.
We assume that our minimal words with $i=0$ are of this form,
and call them minimal normal form.

We claim that this form (minimal normal form) is unique.

To see this, 
we replace all occurrences of $x^{-n}$ in $w=z^i u$ by
$z x^n$ and send this $z$ to the left.
This (possibly not shortest) form is called ``canonical form''.
We claim that this form of the elements in $(p,q)$-torus link group $G$ is unique.
Indeed, suppose that $z^i u$ and $z^j v$ are canonical forms with
$z^{i} u = z^{j} v$ in $G$.
Note that, by introducing the relation $x^p=y^q=z=id$,
we can map $G$ to $H={\mathbb Z}_p * {\mathbb Z}_q$.
Then, we find the equality $u = v$ in $H$, and since
$H$ is a free product, $u$ and $v$ are the same words,
and the equality $u = v$ holds in $G$, too.
It follows that $i = j$,
and the uniqueness of the canonical form is proved.

One can show that the map from the set of minimal normal forms to 
the set of canonical forms is injective by the same arguments
given in 14.1 in \cite{MR2894945}.
It follows that the minimal normal forms are also unique.

\subsubsection{Counting the minimal normal words}

If $z^i u$ is a normal form,
then $u$ is of the form $u_1 v_1 u_2 v_2 \cdots u_r v_r$,
where $u_j\in U_p$ and $v_k\in V_q$ with
$u_1 \neq id$ for $i>1$ and $v_r \neq id$ for $j<r$.
Let $M$ be the set of all minimal normal words.
Define
\begin{align*}
  M_1 & := \{ w\in M \:|\: w=z^i u, i>0\}, &
  M_2 & := \{ w\in M \:|\: w=z^i u, i<0\}, \\
  M_3 & := \{ w\in M \:|\: w=z^0 u, \: u_1 \neq id \}, &
  M_4 & := \{ w\in M \:|\: w=z^0 u, \: u_1 = id \}.
\end{align*}
Then, we have $M=M_1 \cup M_2 \cup M_3 \cup M_4$ and $M_i\cap M_j=\emptyset$
if $i\neq j.$

First, we consider the growth function of $M_1$.
Put $U_p^+ = U_p \setminus \{ x^{-n}\}$.
Then, note that $C_p(t) = 1 + 2t + \dots + 2t^{n-1} + t^n$
(resp. $C_q(t) = 1 + 2t + \dots + 2t^m,$) is the
``growth polynomial'' for $U_p^+$ (resp. $V_q$).
Let $r$ be a positive integer.
Then, the growth polynomial for the set of words 
\[
  M_1(r) := \left\{ u_1 v_1 u_2 v_2 \cdots u_r v_r\; \bigg|\; 
  \begin{array}{ll}
  u_i\in U_p^+, & u_i\neq id\: (\text{for } i>1), \\
  v_j\in V_q,   & u_j\neq id\: (\text{for } j<r)
  \end{array}
 \right\}
\]
is
$C_p(t) \left(\prod_{k=1}^{r-1} (C_q(t)-1)(C_p(t)-1)\right) C_q(t).$
(See, for example, the proof of Prop.\ 1.5 in \cite{MR2894945}.)
Since $x^{-n}$ does not occur in any word of $M_1$, 
any $u$ in $w=z^i u$ of $M_1$ is contained in $M_1(r)$ for some $r$, and 
the growth function $A_1(t)$ for $M_1$ is
\begin{equation*}
\begin{split}
  A_1(t) & = (t+t^2+t^3+\dots) 
             \sum_{r=1}^\infty C_p(t)(C_q(t)-1)^{r-1}(C_p(t)-1)^{r-1}C_q(t)\\
         & = \frac{t}{1-t} \: \cdot \: 
             \frac{C_p(t) C_q(t)}{1-(C_p(t)-1)(C_q(t)-1)}.
\end{split}
\end{equation*}
Since $x^n$ does not occur in any word of $M_2$,
the growth function $A_2(t)$ for $M_2$ is equal to $A_1(t)$,

Next, we consider the growth function $A_3(t)$ for $M_3$. 
Suppose again that $r$ is a positive integer and define
\[
  M_3(r)  = \left\{ u_1 v_1 u_2 v_2 \cdots u_r v_r \in M_3 \: \bigg|\: 
  \begin{array}{ll}
  u_i\in U_p, & u_i\neq id\: (\text{for any } i), \\
  v_j\in V_q, & u_j\neq id\: (\text{for } j<r)
  \end{array}
  \right\}.
\]
Then, each word in $M_3(r)$ is determined by the corresponding word
in $M_1(r)$ (we consider that $x^{-n}$ corresponds to $x^n$ in $U_p^+$.)
and the number, say $k$, of occurrences of $x^{-n}$.
The number $k$ varies between $0$ and
the number of occurrences of $x^n$ in the corresponding word in $M_1(r)$.
Also, the growth series of the image of $M_3(r)$ in $M_1(r)$ is
$(C_p(t)-1)^r(C_q(t)-1)^{r-1}C_q(t)$.
Hence, to calculate the growth series for $M_3(r)$,
for each term in $(C_p(t)-1)^r(C_q(t)-1)^{r-1}C_q(t)$
which corresponds to a word having $k$ occurrences of $x^n$,
we have to multiply this term by $k + 1$.
To do this, put $D_p(t) := 2t + 2t^2 + \dots + 2t^{n-1},$ and $E_p(t) := t^n$.
Since we have
$(C_p(t)-1)^r = \sum_{k=0}^r \binom r k D_p(t)^{r-k} E_p(t)^k,$
and $E_p(t)$ corresponds to $x^n$,
we define
\[
  F^{(r)}(t) := \sum_{k=0}^r (k+1) \binom r k D_p(t)^{r-k} E_p(t)^k.
\]
Then, $F^{(r)}(t)(C_q(t)-1)^{r-1}C_q(t)$ is the growth function for $M_3(r)$.
By the binomial theorem, 
we see that $F^{(r)}(t) = (C_p(t)-1)^r + r E_p(t) (C_p(t)-1)^{r-1}.$
Then, the growth function for $M_3$ is
\begin{equation*}
\begin{split}
A_3(t) & = \sum_{r=1}^\infty F^{(r)}(t)(C_q(t)-1)^{r-1}C_q(t)\\
       & = \frac{(C_p(t)-1)C_q(t)}{1-(C_p(t)-1)(C_q(t)-1)}
         + \frac{E_p(t) C_q(t)}{(1-(C_p(t)-1)(C_q(t)-1))^2}
\end{split}
\end{equation*}
For the last equality, we used $ \sum_{r=1}^\infty r s^{r-1} = 1/(1-s)^2$.

Similarly, the growth function for $M_4$ is
\begin{equation*}
\begin{split}
A_4(t) & = \sum_{r=1}^\infty F^{(r-1)}(t)(C_q(t)-1)^{r-1}C_q(t)\\
       & = \frac{C_q(t)}{1-(C_p(t)-1)(C_q(t)-1)} 
         + \frac{E_p(t) C_q(t)(C_q(t)-1)}{(1-(C_p(t)-1)(C_q(t)-1))^2}.
\end{split}
\end{equation*}
Hence, the growth function for $T_{p,q}'$ is 
\begin{equation*}
\begin{split}
 A(t) & = A_1(t) + A_2(t) + A_3(t) + A_4(t) \\
      & = \frac{1+t}{1-t} \: \frac{C_p(t) C_q(t)}{(C_p(t)+C_q(t)-C_p(t)C_q(t))} 
        + \frac{E_p(t) C_q(t)^2}{(C_p(t)+C_q(t)-C_p(t)C_q(t))^2}
\end{split}
\end{equation*}
which gives the claimed formula.

\subsection{$p$ and $q$ are both even}

Suppose that $p=2n$ and $q=2m$.
We will use the same notation as in the previous section.

Similarly to the previous case, 
we can define the minimal normal form $z^i u$.
However, the case $i=0$ needs some modification.
Suppose that $i=0$.
Let $a \in\{x^n, y^m\}$ and $b \in\{x^{-n}, y^{-m}\}$.
Then, we can interchange the occurrences of $a$ and $b$ in $u$.
Thus, we assume that our minimal words with $i=0$ are of the following types:
\begin{itemize}
\item[($\alpha$)] All terms $x^n$ proceed all terms $x^{-n}$,
           and $u$ contains only $y^{-m}$.
\item[($\beta$)] $u$ contains only $x^n$, 
           and all terms $y^m$ proceed all terms $y^{-m}$.
\end{itemize}
We assume that our minimal words with $i=0$ are of this form.
The uniqueness of this normal form can be proved similarly.
Note that minimal normal words of type
\begin{itemize}
\item[($\gamma$)] $u$ contains only $x^n$ and $y^{-m}$.
\end{itemize}
are exactly the intersection of type ($\alpha$) and ($\beta$).

Let $M$ be the set of all minimal normal words defined above.
Define
\begin{align*}
  M_1 & := \{ w\in M \:|\: w=z^i u,\: i>0\}, \\
  M_2 & := \{ w\in M \:|\: w=z^i u,\: i<0\}, \\
  M_{\alpha} & := \{ w\in M \:|\: w=z^0 u,\: u \text{ is of type ($\alpha$)}\}, \\
  M_{\beta}  & := \{ w\in M \:|\: w=z^0 u,\: u \text{ is of type ($\beta$)}\}, \\
  M_{\gamma} & := \{ w\in M \:|\: w=z^0 u,\: u \text{ is of type ($\gamma$)}\}.
\end{align*}

Since $C_p$ (resp. $C_q$) is the growth polynomial for $U_p^+$ (resp. $V_q^+$),
the growth function $A_1(t)$ for $M_1$ is
\[
  A_1(t)  = \frac{t}{1-t} \: \cdot \: \frac{C_p(t) C_q(t)}{1-(C_p(t)-1)(C_q(t)-1)},
\]
as in the previous case.
The growth function $A_2(t)$ for $M_2$ is equal to $A_1(t)$.
Since $M_{\alpha}$ corresponds to $M_3$ and $M_4$ in the previous case,
its growth function $A_{\alpha}(t)$ is
\begin{equation*}
\begin{split}
  A_{\alpha}(t) & = \sum_{r=1}^\infty F^{(r)}(t) (C_q(t)-1)^{r-1}C_q(t)
           + \sum_{r=1}^\infty F^{(r-1)}(t) (C_q(t)-1)^{r-1}C_q(t)\\
         & = \frac{C_p(t) C_q(t)}{1-(C_p(t)-1)(C_q(t)-1)}
           + \frac{E_p(t) C_q(t)^2}{(1-(C_p(t)-1)(C_q(t)-1))^2}.
\end{split}
\end{equation*}
Similarly, we have
\begin{equation*}
\begin{split}
A_\beta(t) & = \frac{C_p(t) C_q(t)}{1-(C_p(t)-1)(C_q(t)-1)}
         + \frac{E_q(t) C_p(t)^2}{(1-(C_p(t)-1)(C_q(t)-1))^2},\\
A_\gamma(t) & = \frac{C_p(t) C_q(t)}{1-(C_p(t)-1)(C_q(t)-1)}.
\end{split}
\end{equation*}
Then, the rational growth function $A(t)$ is
$A_1(t) + A_2(t) + A_\alpha(t) + A_\beta(t) - A_\gamma(t)$
which gives the claimed formula.

This completes the proof of the main theorem.

\begin{remark}
Besides the proof, we have checked our formula
for many examples using a software called kbmag by Holt \cite{kbmag}.
\end{remark}


\def\cprime{$'$}
\providecommand{\bysame}{\leavevmode\hbox to3em{\hrulefill}\thinspace}
\providecommand{\MR}{\relax\ifhmode\unskip\space\fi MR }
\providecommand{\MRhref}[2]{%
  \href{http://www.ams.org/mathscinet-getitem?mr=#1}{#2}
}
\providecommand{\href}[2]{#2}

\end{document}